\documentclass[12pt]{article}

\usepackage{amssymb, amscd, amsmath, amsfonts, color}
\usepackage{multicol}
\usepackage{verbatim}

\usepackage{amssymb, amscd, amsmath, amsthm, amsfonts,color}
\usepackage{multicol}
\usepackage{verbatim}
\usepackage{authblk}

\usepackage{boxedminipage}
\newenvironment{note}[1][Note]
{\bigskip\begin{center}\begin{boxedminipage}{4.5in}\setlength{\parindent}{1em}\noindent\textbf{#1. }}
{\end{boxedminipage}\end{center}\bigskip}
\newenvironment{todo}{\begin{note}[Todo]}{\end{note}}

\begin{document}

\newtheorem{thm}{Theorem}[section]
\newtheorem*{thm*}{Theorem}
\newtheorem{defn}[thm]{Definition}
\newtheorem*{defn*}{Definition}
\newtheorem{propn}[thm]{Proposition}
\newtheorem*{propn*}{Proposition}
\newtheorem{lemma}[thm]{Lemma}
\newtheorem*{lemma*}{Lemma}
\newtheorem{cor}[thm]{Corollary}
\newtheorem*{cor*}{Corollary}
\newtheorem{eg}[thm]{Example}
\newtheorem*{eg*}{Example}
\newtheorem{claim}[thm]{Claim}
\newtheorem{remark}[thm]{Remark}

\definecolor{redorange}{rgb}{0.800,0.200,0.000}

\title{Primitive orthogonal idempotents for R-trivial monoids}
\author[1,2,3]{Chris Berg}
\author[1,2] {Nantel Bergeron}
\author[1,2] {Sandeep Bhargava}
\author[1,2,3] {Franco Saliola}
\affil[1]{Fields Institute\\ 222 College Street\\ Toronto, ON, Canada}
\affil[2]{York University\\4700 Keele Street\\ Toronto, ON, Canada}
\affil[3]{Universit\'e du Qu\'ebec \`a Montr\'eal, Montr\'eal, QC, Canada}
\maketitle

\newcommand{\bbC}{\mathbb{C}}
\newcommand{\bbF}{\mathbb{F}}
\newcommand{\bbQ}{\mathbb{Q}}
\newcommand{\bbR}{\mathbb{R}}
\newcommand{\bbZ}{\mathbb{Z}}

\newcommand{\calA}{\mathcal{A}}
\newcommand{\calB}{\mathcal{B}}
\newcommand{\calC}{\mathcal{C}}
\newcommand{\calE}{\mathcal{E}}
\newcommand{\calF}{\mathcal{F}}
\newcommand{\calG}{\mathcal{G}}
\newcommand{\calH}{\mathcal{H}}
\newcommand{\calL}{\mathcal{L}}
\newcommand{\calM}{\mathcal{M}}
\newcommand{\calP}{\mathcal{P}}
\newcommand{\calS}{\mathcal{S}}
\newcommand{\calU}{\mathcal{U}}
\newcommand{\calV}{\mathcal{V}}
\newcommand{\calW}{\mathcal{W}}
\newcommand{\calX}{\mathcal{X}}
\newcommand{\calY}{\mathcal{Y}}
\newcommand{\calZ}{\mathcal{Z}}
\newcommand{\frake}{\mathfrak{e}}
\newcommand{\frakg}{\mathfrak{g}}
\newcommand{\mfa}{\mathfrak{a}}
\newcommand{\mfb}{\mathfrak{b}}
\newcommand{\mfc}{\mathfrak{c}}
\newcommand{\mfg}{\mathfrak{g}}
\newcommand{\mfh}{\mathfrak{h}}
\newcommand{\mfi}{\mathfrak{i}}
\newcommand{\mfr}{\mathfrak{r}}
\newcommand{\mfz}{\mathfrak{z}}
\newcommand{\mfA}{\mathfrak{A}}
\newcommand{\mfB}{\mathfrak{B}}

\begin{abstract}
We construct a recursive formula for a complete system of primitive orthogonal
idempotents for any $R$-trivial monoid. This uses the newly proved equivalence
between the notions of $R$-trivial monoid and weakly ordered monoid.
\end{abstract}

\section{Introduction}

Recently, Denton \cite{D} gave a formula for a complete system of primitive orthogonal idempotents for the \emph{$0$-Hecke algebra} of type $A$, the first since the question was raised by Norton \cite{N} in 1979.  A complete system of primitive orthogonal idempotents for \emph{left regular bands} was found by Brown \cite{B} and Saliola \cite{Sa}.  Finding such collections is an important problem in representation theory because they decompose an algebra into projective indecomposable modules:  if $\{e_J \}_{J \in \mathfrak{I}}$ is such a collection for a finite dimensional algebra $A$, then $A = \oplus_{J \in \mathfrak{I}} A e_J$, where each $A e_J$ is a projective indecomposable module. They also allow for the explicit computation of the quiver, the Cartan invariants, and the Wedderburn decomposition of the algebra (see \cite{Bremner, Benson}). For example, in \cite{DHST}, Denton, Hivert, Schilling, and Thi\'ery use a construction of a system of primitive orthogonal idempotents for any $J$-trivial monoid $S$ to derive combinatorially the Cartan matrix and quiver of $S$.

Schocker \cite{S} constructed a class of monoids, called \emph{weakly ordered monoids}, to generalize simultaneously $0$-Hecke monoids and left regular bands, with the broader aim of finding a complete system of orthogonal idempotents for the corresponding monoid algebras.
We achieve this goal here.

A key step is to recognize that the notions of weakly
ordered monoid and \emph{$R$-trivial monoid} are one and the same. This was
first pointed out to us by Thi\'ery \cite{NT} after an intense
discussion between the authors and Denton, Hivert, Schilling, and Thi\'ery. In
Section \ref{rtriv}, we fill out an outline of a proof provided by Steinberg
\cite{St}, who independently made this same observation. In Section
\ref{idempotents}, we use this equivalence to build a recursive formula for a complete system of
primitive orthogonal idempotents for any $R$-trivial monoid. This covers, in particular but not only, the previously known cases of $J$-trivial monoids \cite{DHST} and left regular bands.

\section{Weakly ordered monoids and $R$-trivial monoids}\label{rtriv}

Given any monoid $S$, that is, a set with an associative multiplication and an
identity element, we define a preorder $\leq$ as follows. Given $u,v \in S$,
write $u \leq v$ if there exists $w \in S$ such that $uw=v$. We write $u < v$
if $u \leq v$ but $u \neq v$. Unless stated otherwise, the monoids throughout
the paper are endowed with this ``weak'' preorder. In the monoid theory
literature, the \emph{dual} of this preorder is known as \emph{Green's
$R$-preorder}.

\begin{defn}\label{defn:WOS}
    A finite monoid $S$ is said to be a {\bf weakly ordered monoid} if there is
    a finite upper semi-lattice $(\calL, \preceq)$ together with two maps $C,D:
    S \to \calL$ satisfying the following axioms:
	\begin{enumerate}
		\item
			$C$ is a monoid morphism, i.e. $C(uv) = C(u) \vee C(v)$ for all $u,v \in S$.
		\item
			$C$ is a surjection.
		\item
			If $u,v \in S$ are such that $uv \leq u$, then $C(v) \preceq D(u)$.
		\item
			If $u,v \in S$ are such that $C(v) \preceq D(u)$, then $uv = u$.
	\end{enumerate}
\end{defn}

\begin{remark}\emph{
This notion was introduced by Schocker \cite{S} to generalize $0$-Hecke monoids
and left regular bands, with the broader aim of finding a complete system of
orthogonal idempotents for the corresponding monoid algebras. In his paper, he
actually calls these \textit{weakly ordered semigroups}. However our
understanding is that monoids include an identity element and semigroups do not. So
throughout the paper we call these weakly ordered monoids.
} \end{remark}

\begin{defn}\label{defn:rtrivial}
	A monoid $S$ is {\bf $\boldsymbol{R}$-trivial} if, for all $x,y \in S$, $xS = yS$ implies $ x=y$.
\end{defn}

We restrict our discussion to \emph{finite} $R$-trivial monoids.  


\begin{eg}\label{example:LRB}\emph{
A monoid $S$ is called a {\bf left regular band} if $x^2 = x$ and $xyx = xy$ for all $x,y \in S$.
Left regular bands are $R$-trivial.  Indeed, if $xS = yS$, then there exist $u,v
\in S$ such that $xu=y$ and $x=yv$. But then, since $uv = uvu$, 
\begin{displaymath}
x = yv = xuv = xuvu = yvu = xu = y.
\end{displaymath}
Finitely generated left regular bands are also weakly ordered monoids, see Shocker \cite{S}, e.g. 2.4 and Brown \cite[Appendix B]{B}.
}\end{eg}

\begin{eg}\label{example:0hecke}\emph{
Let $G$ be a Coxeter group with simple generators $\{s_i : i \in I \}$ and relations:
\begin{itemize}
\item $s_i^2 = 1$,
\item $\underbrace{s_is_j s_i s_j \cdots}_{m_{ij}}= \underbrace{s_js_i s_j s_i \cdots}_{m_{ij}}$ for some positive integers $m_{ij}$.
\end{itemize}
 Then the \textbf{0-Hecke monoid} $H^G(0)$ has generators $\{T_i : i\in I\}$ and relations:
\begin{itemize}
\item $T_i^2 = T_i$,
\item $\underbrace{T_iT_j T_i T_j \cdots}_{m_{ij}}= \underbrace{T_jT_i T_j T_i \cdots}_{m_{ij}}$  for some positive integers $m_{ij}$.
\end{itemize}
The weakly ordered monoid $H^{G}(0)$ has maps $C$ and $D$ onto the lattice of subsets of $I$. The map $C$ is the \textit{content} of an element: $C(T_{i_1}T_{i_2}\cdots T_{i_k}) = \{i_1, i_2, \dots, i_k\}$. The map $D$ is the set of right descents of an element: $D(x) = \{ i \in I : xT_i=x\}$.
Note that the preorder for this monoid coincides with the weak order on the elements of the Coxeter group $G$.
\vskip2pt
Of particular interest is the case when $G$ is the symmetric group $\mathfrak{S}_n$. Norton \cite{N} gave a decomposition of the monoid algebra $\mathbb{C}H^{\mathfrak{S}_n}(0)$ into left ideals and classified its irreducible representations. She raised the question of constructing a complete system of orthogonal idempotents for the algebra, which was first answered by Denton \cite{D}.
}\end{eg}

\begin{eg}\emph{
    Let $S$ be the monoid with identity generated by the following matrices:
    \begin{displaymath}
    g_1 := \left[ \begin{array}{rrr} 1 & 0 & 0 \\ 0 & 0 & 1 \\ 0 & 0 & 1 \end{array} \right]
    \quad\text{and}\quad
    g_2 := \left[ \begin{array}{rrr} 0 & 1 & 0 \\ 0 & 1 & 0 \\ 0 & 0 & 1 \end{array} \right].
    \end{displaymath}
    Then $S = \{ 1,\  g_1,\  g_2,\  g_1 g_2,\  g_2 g_1\}$ and $S$ is both an
    $R$-trivial monoid and a weakly ordered monoid.  For example, we can take
    $\calL$ to be usual lattice of subsets of $\{1,2\}$, with $C: S \to
    \calL$ given by
		\begin{displaymath}
			C(1) = \emptyset, \ C(g_1) = \{1\}, \ C(g_2) = \{2\}, \ C(g_1 g_2) = C(g_2 g_1) = \{1,2\},
		\end{displaymath}
and $D: S \to \calL$ given by		
		\begin{displaymath}
			D(1) = \emptyset, \ D(g_1) = \{1\}, \ D(g_2) = D(g_1 g_2) = \{2\}, \ D(g_2 g_1) = \{1,2\}.
		\end{displaymath}
The monoid $S$, however, is neither a left regular band, since $g_1 g_2$ is not idempotent, nor isomorphic to the $0$-Hecke monoid $H^{G}(0)$ on two generators, since the latter always has an even number of elements.
}\end{eg}

The fact that the above examples are all weakly ordered and $R$-trivial is no
coincidence:  the purpose of this section is to show that these two notions are
equivalent.

\begin{remark}\label{propn:Rtrivial iff preorder=partialorder}
 A monoid $S$ is $R$-trivial if and only if the preorder $\leq$ defined above is a partial order.
\end{remark}
\begin{proof}
	Suppose $S$ is an $R$-trivial monoid and $x,y \in S$ are such that $x \leq y$ and $y \leq x$.  Then there exist $u,v \in S$ such that $xu=y$ and $yv=x$.  So $y \in xS$ and $x \in yS$, implying that $yS \subseteq xS$ and $xS \subseteq yS$.  That is, $xS = yS$.  Since $S$ is $R$-trivial, $x = y$.

	On the other hand, suppose that the given preorder is a partial order, and that $xS = yS$ for some $x,y \in S$.  Since $x = x \cdot 1 \in xS = yS$, we have that $x = yu$ for some $u \in S$.  So $y \leq x$.  Similarly, $y \in xS$ implies that $x \leq y$.  The antisymmetry of $\leq$ implies then that $x=y$.  So $S$ is $R$-trivial.
\end{proof}

\begin{cor}\label{cor:WOS is RtrivialMonoid}
	A weakly ordered monoid is an $R$-trivial monoid.
\end{cor}
\begin{proof}
	Let $S$ be a weakly ordered monoid.  Lemma $2.1$ in \cite{S} shows that the defining conditions of a weakly ordered monoid imply that the preorder on $S$ is a partial order.  The result now follows from Proposition \ref{propn:Rtrivial iff preorder=partialorder}.
\end{proof}

    We will show that any finite $R$-trivial monoid $S$ is a weakly
ordered monoid using an argument outlined by Steinberg \cite{St}.
We must establish the existence of an upper semi-lattice $\calL$ and two maps
$C$ and $D$ from $S$ to $\calL$ that satisfy the conditions of Definition
\ref{defn:WOS}. We gather here the definitions of $\calL$, $C$ and $D$:
\begin{enumerate}
\item 
$\calL$ is the set of left ideals $Se$ generated by idempotents $e$ in $S$,
ordered by reverse inclusion;
\item
$C : S \to \calL$ is defined as $C(x) = Sx^\omega$,
where $x^\omega$ is the idempotent power of $x$ (see Lemma \ref{lemma:xox=xo});
\item
$D : S \to \calL$ is defined as $D(u) = C(e)$, where $e$ is some maximal element
in the set $\{s \in S : us = u\}$ (with respect to the preorder $\leq$).
\end{enumerate}

The following lemma is a simple statement about $R$ trivial monoids which is used frequently throughout the paper.

\begin{lemma}\label{rtrivial}
Suppose $S$ is an $R$-trivial monoid.
If $x,y,z \in S$ are such that $xyz=x$, then $xy=x$.

Consequently, if $x, y_1, y_2, \dots, y_m \in S$ are such that
$xy_1\cdots y_m = x$, then $xy_i = x$ for all $1 \leq i \leq m$.
\end{lemma}

\proof
If $xyz = x$ then $xyS = xS$. Therefore $xy=x$ by the definition of $S$ being $R$-trivial.
The second statement immediately follows from the first.
\endproof

The remainder of this section is dedicated to showing that these objects are
well defined and that they satisfy the conditions of Definition \ref{defn:WOS}.
We begin by recalling some classical results from the monoid literature.
The following is \cite[Proposition 6.1]{Pin2009}.

\begin{lemma}\label{lemma:xox=xo}
    If $S$ is a finite monoid, then for each $x \in S$, there exists a
    positive integer $\omega = \omega(x)$ such that $x^\omega$ is idempotent,
    i.e. $(x^\omega)^2 = x^\omega$. Furthermore, if $S$ is $R$-trivial, then we
    also have $x^{\omega}x = x^{\omega}$.
\end{lemma}

\proof
    Consider the elements $x, x^2, x^3, \dots$. Since $S$ is finite, there
    exist positive integers $i$ and $p$ such that $x^{i+p} = x^i$.
    Then $x^{k+p} = x^{k}$ for all $k \geq i$, so if we take $\omega = ip$,
    then $(x^\omega)^2 = x^{\omega + ip} = x^{\omega}$.

    If $S$ is $R$-trivial, then $x^\omega \leq x^\omega x \leq x^\omega
    x^\omega = x^\omega$, and so $x^\omega x = x^\omega$.
\endproof

\begin{remark}\emph{
In what follows, if $x \in \mathbb{C}S$ and there exists $N$ such that $x^{N+1} = x^N$, we sometimes abuse notation by writing $x^\omega$ in place of $x^N$.
} \end{remark}

\begin{lemma}
\label{lemma:omega}
Let $S$ be a finite $R$-trivial monoid. 
For all $x$ and $y$ in $S$, 
\begin{multicols}{2}
\begin{enumerate}
\item\label{item_i} $(xy)^{\omega} x = (xy)^{\omega}$;
\item\label{item_i y} $(xy)^\omega y = (xy)^\omega$;
\item\label{item_ii} $(xy)^{\omega} x^{\omega} = (xy)^{\omega}$;
\item\label{item_ii y} $(xy)^\omega y^\omega = (xy)^\omega$;
\item\label{item_iii} $(x^{\omega} y^{\omega})^{\omega} x^{\omega} = (x^{\omega} y^{\omega})^{\omega}$;
\item\label{item_iv} $(x^{\omega} y^{\omega})^{\omega} = (x^{\omega} y^{\omega})^{\omega}\, (xy)$;
\item\label{item_v} $(x^{\omega} y^{\omega})^{\omega} = (x^{\omega} y^{\omega})^{\omega}\, (xy)^{\omega}$.
\item[]
\end{enumerate}
\end{multicols}
\end{lemma}
\begin{proof}
(\ref{item_i}) Since $(xy)^{\omega}x \in (xy)^{\omega}S$, it follows that
$(xy)^{\omega}x S \subseteq (xy)^{\omega}S$.  To show the reverse inclusion,
note that $(xy)^{\omega} = (xy)^{\omega}(xy) = \big( (xy)^{\omega}x \big) y \in
(xy)^{\omega} x S$, where the first equality follows from Lemma
\ref{lemma:xox=xo}.  So $(xy)^{\omega} S \subseteq (xy)^{\omega} x S$.  Thus
$(xy)^{\omega} x S = (xy)^{\omega} S$.  Since $S$ is an $R$-trivial monoid, the
desired result follows.

(\ref{item_i y}) Apply (\ref{item_i}) and Lemma \ref{lemma:xox=xo}:
\[(xy)^{\omega}  = (xy)^{\omega} (xy)  = \Big( (xy)^{\omega} x \Big) y = (xy)^{\omega} y. \]

(\ref{item_ii}) This follows from applying (\ref{item_i}) repeatedly.

(\ref{item_ii y}) This follows from applying (\ref{item_i y}) repeatedly.

(\ref{item_iii}) Let $u=x^{\omega}$ and
$v=y^{\omega}$.  Now, by (\ref{item_i}), $(uv)^{\omega} u =
(uv)^{\omega}$.

(\ref{item_iv}) We compute:
	\begin{align*}
		(x^{\omega} y^{\omega})^{\omega} & = (x^{\omega} y^{\omega})^{\omega-1} x^{\omega} y^{\omega} \\
		& = (x^{\omega} y^{\omega})^{\omega-1} x^{\omega} y^{\omega} y & \text{(by Lemma \ref{lemma:xox=xo})}\\
		& = (x^{\omega} y^{\omega})^{\omega} y \\
		& = (x^{\omega} y^{\omega})^{\omega} x^{\omega} y & \text{(by (\ref{item_iii}))}\\
		& = (x^{\omega} y^{\omega})^{\omega} x^{\omega} x y & \text{(by Lemma \ref{lemma:xox=xo})}\\
		& = (x^{\omega} y^{\omega})^{\omega} x y & \text{(by (\ref{item_iii}))}
	\end{align*}

(\ref{item_v}) This follows by repeatedly applying part (\ref{item_iv}).
\end{proof}

	We are now ready to construct a lattice corresponding to the $R$-trivial monoid $S$.  Define
		\begin{displaymath}
			\calL := \{ Se \, : \ e \in S \text{ such that } e^2 = e \}.
		\end{displaymath}
    That is, $\calL$ is the set of left ideals generated by the idempotents of
    $S$.  Define a partial order on $\calL$ by
		\begin{displaymath}
			Se \preceq Sf \iff Se \supseteq Sf.
		\end{displaymath}

\begin{propn}
    \label{propn:S(ef)o=lubOfSeAndSf}
	If $e, f$ are idempotents in $S$, then $S(ef)^{\omega}$ is the least upper bound of $Se$ and $Sf$ in $\calL$.
\end{propn}
\begin{proof}
	First, let us show that $S(ef)^{\omega}$ is an upper bound for $Se$ and $Sf$.  Since, by Lemma \ref{lemma:omega}\,(1), $(ef)^{\omega} = (ef)^{\omega} e$, we have that $(ef)^{\omega} \in Se$.  Hence $S(ef)^{\omega} \subseteq Se$ and $S(ef)^{\omega} \succeq Se$.  Moreover, $(ef)^{\omega} = \Big((ef)^{\omega-1}e\Big)f \in Sf$.  So $S(ef)^{\omega} \subseteq Sf$ and $S(ef)^{\omega} \succeq Sf$.  So $S(ef)^{\omega}$ is an upper bound for $Se$ and $Sf$.  
	
	Next, let us show that $S(ef)^{\omega}$ is the least upper bound for $Se$ and $Sf$.  Suppose $g$ is an idempotent in $S$ such that $Sg$ is an upper bound for $Se$ and $Sf$.  That is, $Sg \subseteq Se$ and $Sg \subseteq Sf$.  Since $Sg \subseteq Se$, $g = te$ for some $t \in S$.  But then $ge = (te)e = te^2 = te = g$.  Similarly, $Sg \subseteq Sf$ implies that $gf = g$.  So $g(ef) = (ge)f = gf = g$ and it follows that
		\begin{displaymath}
			g = g(ef) = \Big(g(ef)\Big)(ef) = g (ef)^2 = \Big(g(ef)\Big) (ef)^2 = g(ef)^3 = \cdots = g(ef)^{\omega}.
		\end{displaymath}
	Consequently, $g \in S(ef)^{\omega}$, $Sg \subseteq S(ef)^{\omega}$, and $Sg \succeq S(ef)^{\omega}$.  So $S(ef)^{\omega}$ is the least upper bound of $Se$ and $Sf$.
\end{proof}

As a result, we may define the join of two elements $Se$ and $Sf$ in $\calL$ by 
	\begin{displaymath}
		Se \vee Sf = S(ef)^{\omega}.
	\end{displaymath}
That is, $\calL$ is an upper semi-lattice with respect to this join operation.
This observation proves the following.

\begin{propn}\label{propn:C is surjective morphism}
    The map $C: S \to \calL$ defined by $C(x) = Sx^{\omega}$
    is a surjective monoid morphism.
\end{propn}

\begin{proof}
	Let $x,y \in S$.  By Lemma \ref{lemma:omega}\,(5), we know that $(x^{\omega} y^{\omega})^{\omega} = (x^{\omega} y^{\omega})^{\omega} (xy)^{\omega}$.  Hence, $(x^{\omega} y^{\omega})^{\omega} \in S(xy)^{\omega}$ and $S(x^{\omega} y^{\omega})^{\omega} \subseteq S(xy)^{\omega}$.  
	
	To show the reverse inclusion, we begin by noting that, by Lemma \ref{lemma:omega}\,(2), $(xy)^{\omega} = (xy)^{\omega} x^{\omega}$.  So $(xy)^{\omega} \in S x^{\omega}$ and $S(xy)^{\omega} \subseteq S x^{\omega}$.  That is, $S(xy)^{\omega} \succeq S x^{\omega}$.

 Lemma \ref{lemma:omega}\,(4), implies that $(xy)^{\omega} \in S y^{\omega}$, which implies that $S(xy)^{\omega} \subseteq S y^{\omega}$ and $S(xy)^{\omega} \succeq S y^{\omega}$.  In particular, $S (xy)^{\omega}$ is an upper bound for both $S x^{\omega}$ and $S y^{\omega}$.  So $S (xy)^{\omega}	\succeq S x^{\omega} \vee S y^{\omega} = S (x^{\omega} y^{\omega})^{\omega}$, that is, $S(xy)^{\omega} \subseteq S(x^{\omega} y^{\omega})^{\omega}$.
	
	Thus $C(xy) = S(xy)^{\omega} = S(x^{\omega} y^{\omega})^{\omega} = S x^{\omega} \vee S y^{\omega} = C(x) \vee C(y)$, and $C$ is a monoid morphism.  Finally, we know that every element of $\calL$ is of the form $Se$ for some idempotent $e$ in $S$.  But then $C(e) = Se^{\omega} = Se$; that is, $C$ is a surjective morphism.
\end{proof}

Here is an alternate and useful characterization of $C(x)$.
\begin{propn}\label{propn:C(x)=aInSstax=a}
	$C(x) = \{ a \in S \, : \ ax=a \}$ for all $x \in S$.
\end{propn}
\proof
	Take an arbitrary element in $C(x) = Sx^{\omega}$, say $tx^{\omega}$.  Since $\big( tx^{\omega} \big) x = t \big( x^{\omega} x \big) = t x^{\omega}$ by Lemma \ref{lemma:xox=xo}, we see that $tx^{\omega} \in \{ a \in S \, : \ ax=a \}$.  On the other hand, take $b \in \{ a \in S \, : \ ax =a \}$.  Then 
		\begin{displaymath}
			bx^{\omega} = (bx) x^{\omega-1} = bx^{\omega-1} = (bx) x^{\omega-2} = b x^{\omega-2} = \cdots = bx = b.
		\end{displaymath}
	Therefore, $b \in Sx^{\omega}$.
\endproof

We now define the map $D: S \to \calL$.  Given $u \in S$, let $D(u) = C(e)$, where $e$ is a maximal element in the set $\{ s \in S \, : \ us = u \}$.  
To check that $D$ is well defined, let $e$ and $f$ be two distinct maximal elements in $\{ s \in S \, : \ us = u \}$.  Since $e \leq ef$ and $u(ef) = (ue)f = uf = u$, by the maximality of $e$, $e = ef$.  Similarly, since $f \leq fe$ and $u(fe) = u$, the maximality of $f$ implies $f = fe$.  Then, by Proposition \ref{propn:C is surjective morphism}, 
	\begin{displaymath}
		C(e) = C(ef) = C(e) \vee C(f) = C(f) \vee C(e) = C(fe) = C(f).
	\end{displaymath}
Note that the maximality of $e$ and $ue^2 = u$ also implies that $e = e^2$, that is, $e$ is idempotent.

The next proposition shows that the maps $C$ and $D$ interact in precisely the manner given in conditions $2$ and $3$ in Definition \ref{defn:WOS}.  The following lemma will help us prove this proposition.

\begin{lemma}\label{lemma:ifSGobblesHigherThenSGobblesLower}
	Let $x,y \in S$.  If $x \leq y$, then $C(x) \preceq C(y)$.
\end{lemma}
\proof
			If $s \in C(y)$, then $sy = s$.  Since $x \leq y$, there exists $t \in S$ such that $y = xt$.  So $sxt = s$, implying $sx \leq s$.  That is, $s \in C(x)$.  Hence $C(y) \subseteq C(x)$, or $C(x) \preceq C(y)$ since $s \leq sx$ and $S$ is $R$-trivial.
\endproof

\begin{propn}\label{propn:C and D satisfy WOS conditions}
Let $u, v \in S$. 
(i) If $uv \leq u$, then $C(v) \preceq D(u)$.
(ii) If $C(v) \preceq D(u)$, then $uv = u$.
\end{propn}
\proof 
	(i) \ Since $u \leq uv$, $u = uv$.  Hence $v$ lies in the set $\{ s \in S \, : \ us = u \}$.  Let $e$ be a maximal element in this set such that $v \leq e$.  Then, by Lemma \ref{lemma:ifSGobblesHigherThenSGobblesLower}, $C(v) \preceq C(e) = D(u)$.
\vskip3pt	
	(ii) By definition, $D(u) = C(e)$, where $e$ is a maximal element of $\{ s \in S \, : \ us = u \}$.  So if $C(v) \preceq D(u)$, then $C(v) \preceq C(e)$.  Hence $C(e) \subseteq C(v)$.  Since $ue = u$, $u$ lies in $C(e)$.  So $u$ is also a member of $C(v)$; that is, $uv = u$.  
\endproof

Propositions \ref{propn:C is surjective morphism} and \ref{propn:C and D satisfy WOS conditions} tell us that an $R$-trivial monoid is a weakly ordered monoid.  Combining this with Corollary \ref{cor:WOS is RtrivialMonoid}, we have the following result.

\begin{thm}
	A finite monoid $S$ is a weakly ordered monoid if and only if it is an $R$-trivial monoid. 
\end{thm}



\section{Constructing idempotents}\label{idempotents}


\begin{defn}
Let $A$ be a finite dimensional algebra with identity $1$. We say that a set of
nonzero elements $\Lambda = \{ e_J : J \in \mathcal{I}\}$ of $A$ is a {\bf
complete system of primitive orthogonal idempotents for $A$} if:
\begin{enumerate}
\item\label{Aidemp} 
    each $e_J$ is \emph{idempotent}:
    that is, $e_J^2 = e_J$ for all $J \in \mathcal{I}$;
\item\label{Aortho}
    the $e_J$ are pairwise \emph{orthogonal}:
    $e_J e_K = 0$ for $J,K \in \mathcal{I}$ with $J \neq K$;
\item\label{Aprim}
    each $e_J$ is \emph{primitive}
    (meaning that it cannot be further decomposed into orthogonal idempotents):
    if $e_J = x+y$ with $x$ and $y$ orthogonal idempotents in $A$,
    then $x=0$ or $y=0$;
\item\label{Acomp}
    $\{e_J : J \in \mathcal{I}\}$ is \emph{complete} (meaning that the elements
    sum to the identity): $\sum_{J \in \mathcal{I}} e_J = 1$.
\end{enumerate}
\end{defn}

\begin{remark}\emph{
    If $\Lambda$ is a \emph{maximal} set of nonzero elements satisfying
    conditions (\ref{Aidemp}) and (\ref{Aortho}), then $\Lambda$ is a complete
    system of primitive orthogonal idempotents (that is, (\ref{Aprim}) and
    (\ref{Acomp}) also hold). Indeed, $e_J$ is primitive, for if $e_J$ could be
    written as $x+y$, then we could replace $e_J$ in $\Lambda$ with $x$ and $y$,
    contradicting the maximality of $\Lambda$. To see (\ref{Acomp}), we just note
    that if $\sum_K e_K \neq 1$, then $1-\sum_K e_K$ is idempotent and
    orthogonal to all other $e_K$. Combining this element with $\Lambda$ would
    again contradict the maximality of $\Lambda$.
} \end{remark}

\vspace{.2in}
Let $S$ denote a finite weakly ordered monoid with $C$ and $D$ being the associated ``content'' and ``descent'' maps from $S$ to an upper semi-lattice $\mathcal{L}$. We let $\mathcal{G}$ denote a set of generators of $S$.
The main goal of this paper is to build a method for finding a complete system of orthogonal idempotents for the monoid algebra $\mathbb{C} S$. In particular, this solves the problem posed by Norton about the $0$-Hecke algebra for the symmetric group.


For each $J \in \mathcal{L}$, we define a \textbf{Norton element} $A_J T_J$.  Let us begin by defining $T_J$:

 \[T_J = \Big(\prod_{\substack{g \in \mathcal{G}\\ C(g) \preceq J}}  g^\omega\Big)^\omega \in S.\]

\begin{remark}\emph{
A different ordering of the set $\calG$ of generators may produce different $T_J$'s; so we fix an (arbitrarily chosen) order uniformly for all $J$. 
} \end{remark}

We now define the $A_J$ in the Norton element $A_J T_J$.  First we let \[B_J = \prod_{\substack{g \in \mathcal{G}\\C(g) \not \preceq J}} (1-g^\omega) \in \mathbb{C}S.\]

In the spirit of Lemma \ref{lemma:xox=xo},
we would like to raise $B_J$ to a sufficiently high power so that it is idempotent. However, $B_J$ is not an element of the monoid $S$, so $(B_J)^\omega$ may not be well defined. The following lemma and corollary shows that it actually is.

\begin{defn}
Given $x=\sum_{w\in S} c_w w \in \mathbb{C}S$, the \textbf{coefficient} of $w$ in $x$ is $c_w$. We say that $w$ is a \textbf{term} of $x$ if the coefficient of $w$ in $x$ is nonzero.
\end{defn}

\begin{lemma}\label{lemma:increasing}
Let $b \in S$ and suppose $bx^\omega = b$ for some $x \in \mathcal{G}$ with $C(x) \not \preceq J$. If $c$ is a term of $bB_J$, then $c>b$. 
\end{lemma}

\proof
Let $\mathcal{D} = \{ x^\omega : x \in \mathcal{G}, C(x) \not \preceq J, bx^\omega=b \}$. By assumption $\mathcal{D}$ is not empty. Let  $g_1, g_2, \dots, g_m$ be the generators which appear in the definition of $B_J$. Then \[B_J = \sum_{i_1<i_2<\dots <i_k} (-1)^k g_{i_1}^\omega g_{i_2}^\omega\cdots g_{i_k}^\omega.\] 
It follows from Lemma \ref{rtrivial} that
the coefficient of $b$ in $bB_J$ is counting the terms in $B_J$ where each of $g_{i_1}, \dots, g_{i_k}$ come from $\mathcal{D}$, weighted with sign $(-1)^k$. If $|\mathcal{D}|=m\geq1$ then this is $1-m+\binom{m}{2}-\binom{m}{3} + \dots + (-1)^m = 0$. Therefore $c \neq b$.
The statement now follows from the definition of order, as every term $c$ of $bB_J$ must be of the form $c = bz$ for some term $z$ appearing in $B_J$, and hence $c \geq b$.
\endproof

\begin{lemma}\label{lemma:yoBJ = 0}
	For every $J \in \calL$, there exists an integer $N$ such that $y^{\omega} B_J^N = 0$ for all $y \in \calG$ with $C(y) \npreceq J$.  
\end{lemma}
\proof
Let $N = \ell+1$, where $\ell$ is the length of
the longest chain of elements in the poset $(S, \leq)$. 

Suppose $y^\omega B_J^N \neq 0$.  Let $c_N$ be a term of $B_J^{N}$. Then
$c_N$ is a term of $c_{N-1}B_J$ for some term $c_{N-1}$ in $y^\omega
B_J^{N-1}$. Since $y^\omega y^\omega = y^\omega$, Lemma
\ref{lemma:increasing} implies that $y^\omega$ is not a term of $y^\omega
B_J^k$ for any $k\geq1$, so that $c_{N-1} = y^\omega g_1^\omega\cdots
g_m^\omega$ for some $m\geq1$ and $g_i\in\calG$ with $C(g_i)\not\preceq J$. In
particular, $c_{N-1}g_m^\omega = c_{N-1}$, and so, again by Lemma
\ref{lemma:increasing}, $c_{N} > c_{N-1}$. 
Repeated application of this argument produces a decreasing chain \[c_N > c_{N-1} > c_{N-2} > \cdots > c_1\]
of elements in $S$, contradicting the fact that the length of the
longest chain of elements in $(S,\leq)$ is $\ell$.
\endproof

\begin{cor}\label{cor:finiteN}
For every $J \in \calL$ there exists an $N$ such that $B_J^{N+1} = B_J^N$.
\end{cor}
\proof
By Lemma \ref{lemma:yoBJ = 0}, $(B_J-1)B_J^N = 0$ for a sufficiently large $N$ since every element of $B_J-1$ is of the form $\alpha y^{\omega}$ where $\alpha \in \mathbb{C}$, $y \in \calG$ and $C(y) \npreceq J$.
\endproof

\begin{remark}
Corollary \ref{cor:finiteN} is a special property of an $R$-trivial monoid, and is not true for a general monoid. For instance if an element $x$ of a semigroup $S$ generates a finite cyclic group of order $2$, then $(1-x)^k = 2^{k-1} - 2^{k-1}x$, so $(1-x)^{k+1} \neq (1-x)^k$ for all $k$. 
\end{remark}

This now allows us to define $A_J = B_J^\omega$.

\begin{lemma}\label{lemma:AT}
Let $J \in \calL$. Then:
\begin{enumerate}
\item\label{Tx}
$T_J x= T_J $ for all $x$ such that $C(x) \preceq J$; 
\item\label{yA} $y^\omega A_J = 0$ for all $y$ such that $C(y) \not \preceq J$ and $y \in \mathcal{G}$.
\end{enumerate}
\end{lemma}

\proof
Since $J = C(T_J)$, $C(x) \preceq J$ implies $C(x) \supseteq C(T_J)$.  We also know that $T_J \in C(T_J)$ because $T_J$ is idempotent.  So $T_J \in C(x)$, that is, $T_J x = T_J$.

The second part follows from Lemma \ref{lemma:yoBJ = 0} since $A_J = B_J^N$.
\endproof

\begin{remark}\emph{ Although $T_J$ and $A_J$ are idempotents individually, their product, the Norton element $z_J$, need not be.  For example, take the 0-Hecke algebra $H_6(0)$ corresponding to the symmetric group $\mathfrak{S}_6$. Let $J$ be the subset $\{1,4,5\}$ of $\{1,2,3,4,5\}$. Then  $T_J = T_1T_4T_5T_4$, $A_J = (1-T_2)(1-T_3)(1-T_2)$ and $z_J$ is their product. No power of $z_J$ is idempotent.
} \end{remark}


\begin{lemma}\label{z!=0}
The coefficient of $T_J$ in $z_J = A_J T_J$ is 1. All other terms $y$ in $z_J$ have $C(y) \succ J$.
\end{lemma}

\proof
The coefficient of the identity element $1$ in $A_J$ is 1. Each term of $A_JT_J$ is of the form $aT_J$ for a term $a$ of $A_J$. If $a \neq 1$, then $C(a) \npreceq J$ so $C(aT_J) = C(a) \vee C(T_J) \succ C(T_J) = J$. Hence the coefficient of $T_J$ in $A_JT_J$ is 1 and all other terms have content greater than $J$.
\endproof

\begin{lemma}\label{lemma:z-ortho}
If $J \not \preceq K$ then $z_J z_K = 0$.
\end{lemma}

\proof
Since $J \not \preceq K$, there exists a $g \in \mathcal{G}$ with $C(g) \preceq J$ but $C(g) \not \preceq K$. Then, using Lemma \ref{lemma:AT} (\ref{Tx}) and Lemma \ref{lemma:AT} (\ref{yA}),
$z_Jz_K = A_JT_JA_KT_K =
A_J(T_J g^\omega) A_K T_K =
A_JT_J (g^\omega A_K) T_K = 0.
$
\endproof

\begin{lemma}\label{lemma:eventuallyzero}
For all $J \in \calL$, there exists an $N$ such that $ \left( 1- z_J \right)^N z_J^2= 0$.
\end{lemma}

\begin{proof}
To simplify the notation, let us temporarily set $T=T_J$, $A=A_J$ and $z=z_J=AT$.
We first note that for any integer $k \geq 0$,
\begin{align*}
(1-z)^kz^2 
&= z (1-z)^k z\\
& =AT(1-AT)^kAT\\
& = A(T(1 - A)T)^kAT.
\end{align*}
We will show that $ (T(1 - A)T)^NA = 0$ for $N>\ell$, 
where $\ell$ is the length of the longest chain in the poset $(S,\leq)$.

Let us write $1-A = \sum_{a\in S} c_a a$ where each term has $c_a \neq 0$ only if $a = g_1^\omega \cdots g_k^\omega$ with $C(g_i) \not \preceq J$ for all $i$. Therefore 
\begin{displaymath}
T(1-A)T = \sum_{a \in S} c_a TaT
        = \sum_{a \in S \atop TaT = Ta} c_a Ta \ \  
        + \sum_{a \in S \atop TaT \neq Ta} c_a TaT.
\end{displaymath}
Note that $c_1 = 0$ since $1$ is not a term of $(1-A)$.
If $TaT = Ta$, then we have
\begin{displaymath}
TaT \cdot (T(1-A)T) = Ta(1-A)T = Ta - TaAT = Ta
\end{displaymath}
since $aA = 0$ by Lemma \ref{lemma:AT}. Thus,
\begin{align*}
(T(1-A)T)^N  
        &= \left(\sum_{a_1 \in S \atop Ta_1T = Ta_1} c_{a_1} Ta_1
        + \sum_{a_1 \in S \atop Ta_1T \neq Ta_1} c_{a_1} Ta_1T\right) (T(1-A)T)^{N-1} \\
        &= \sum_{a_1 \in S \atop Ta_1T = Ta_1} c_{a_1} Ta_1
        + \left(\sum_{a_1 \in S \atop Ta_1T \neq Ta_1} c_{a_1} Ta_1T\right) (T(1-A)T)^{N-1}.
\end{align*}
Next, rewrite the second summand above using the same argument:
\begin{align*}
& \left(\sum_{a_1 \in S \atop Ta_1T \neq Ta_1} c_{a_1} Ta_1T\right) (T(1-A)T)^{N-1} \\
=& \left(\sum_{a_1 \in S \atop Ta_1T \neq Ta_1} c_{a_1} Ta_1T\right) 
  \left(\sum_{a_2 \in S} c_{a_2} Ta_2T\right) (T(1-A)T)^{N-2} \\
=& \left(\sum_{a_1,a_2 \in S \atop Ta_1T \neq Ta_1} c_{a_1} c_{a_2} Ta_1Ta_2T \right)
  (T(1-A)T)^{N-2} \\
=& \sum_{Ta_1T \neq Ta_1 \atop Ta_1Ta_2T = Ta_1Ta_2} c_{a_1} c_{a_2} Ta_1Ta_2 \\
   & \qquad\qquad + \left(\sum_{Ta_1T \neq Ta_1 \atop Ta_1Ta_2T \neq Ta_1Ta_2} c_{a_1} c_{a_2} Ta_1Ta_2T \right)
  (T(1-A)T)^{N-2}.
\end{align*}
Continuing in this way, we can write $(T(1-A)T)^N$ in the form
\begin{align*}
(T(1-A)T)^N = 
\left(\sum_{} c_{a_1} Ta_1
+ \cdots +
\sum_{} c_{a_1}\cdots c_{a_N} Ta_1\cdots Ta_N\right) \qquad
\\
+
\sum_{Ta_1\cdots Ta_iT \neq Ta_1\cdots Ta_i \atop 1\leq i \leq N} c_{a_1}\cdots c_{a_N} Ta_1\cdots Ta_NT.
\end{align*}


By Lemma \ref{lemma:AT}, we have $a_i A = 0$ for all terms $a_i$ in $1-A$, and so
\begin{align*}
(T(1-A)T)^N \cdot A = 
\left(\sum_{Ta_1\cdots Ta_iT \neq Ta_1\cdots Ta_i \atop 1\leq i \leq N} c_{a_1}\cdots c_{a_N} Ta_1\cdots Ta_NT\right) A.
\end{align*}
This summation is $0$ as it ranges over an empty set: indeed, if it
is not empty, we would have an increasing chain of length $N > \ell$, namely

\begin{displaymath}
	Ta_1 \ < \ Ta_1Ta_2 \ < \ Ta_1Ta_2Ta_3 \ < \ \cdots \ < \ Ta_1Ta_2 \cdots Ta_N,
\end{displaymath}

Therefore, $(T(1-A)T)^NA = 0$.
\end{proof}

\begin{defn}
Let $J \in \calL$. Let 
\begin{displaymath}
P_J :=  
\sum_{n,m\geq 0} \left( 1 - z_J \right)^{n+m} z_J^2
\ = \
\sum_{k\geq 0} (k+1) \left( 1 - z_J \right)^{k} z_J^2.
\end{displaymath}
(In Remark \ref{PJsummationfree} we establish a summation-free formula for $P_J$.)
\end{defn}

\begin{remark}\emph{
Lemma \ref{lemma:eventuallyzero} shows there are only finitely many terms
in the summation of $P_J$.
Therefore $P_J$ is a well defined element of $\bbC S$ for
each $J \in \calL$.
} \end{remark}

\begin{remark}\emph{\label{jtriv}
A monoid $S$ is called $J$-trivial if $SxS = SyS$ implies $x=y$ for all $x,y \in S$. When $S$ is $J$-trivial it suffices to define \[P_K = \ \sum_{n\geq 0} ( 1 - z_K)^{n}  z_K.\] 
} \end{remark}

\begin{lemma}\label{p!=0}
The coefficient of $T_J$ in $P_J$ is $1$ and all other terms $y$ of $P_J$ have $C(y) \succ J$.
\end{lemma}

\proof
If $n+m > 0$ then, using that $T_J$ is idempotent, \[A_JT_JA_JT_J (1-A_JT_J)^{n+m} = A_JT_JA_J (T_J-T_JA_JT_J)^{n+m}.\]
Each term $x$ in $(T_J - T_JA_JT_J)^{n+m}$ has $C(x) \succ J$, so no $T_J$ appears in  $z_J^2 (1-z_J)^{n+m}$. The coefficient of $T_J$ in $z_J$ is $1$, by Lemma \ref{z!=0}. Hence $T_J$ appears in $z_J^2  (1-z_J)^0$ with coefficient $1$. By Lemma \ref{z!=0}, since all of the terms $y \neq T_J$ of $z_J$ have $C(y) \succ J$ and $P_J$ is a polynomial in $z_J$, all other terms $w$ of $P_J$ must have $C(w) \succ J$.
\endproof

\begin{remark}\emph{\label{polys} As polynomials in $x$ we have
for any nonnegative integer $N$:
\[ x\sum_{n=0}^N (1-x)^n  = 1-(1-x)^{N+1}.\]
} \end{remark}

\begin{propn}\label{lemma:p-idemp}
	For each $J \in \calL$, the element $P_J$ is idempotent.
\end{propn}
\proof
Let $J\in \mathcal{L}$ be fixed and let $N$ be such that  $(1-z_J)^N z_J^2= 0$. 
Let us temporarily denote $z_J$ by $z$. We can use Lemma \ref{polys} to rewrite
$P_J$ as
\begin{align*} 
P_J 
&= \sum_{n, m \geq 0} z^2 (1 - z)^{n+m} = \sum_{n=0}^N \sum_{m=0}^{N-n} z^2 (1 - z)^{n+m} \\
&= \sum_{n=0}^N (1-z)^{n} \left(z^2 \sum_{m=0}^{N-n} (1 - z)^{m}\right) = \sum_{n=0}^N (1-z)^{n} \left(z - z(1 - z)^{N-n+1}\right) \\
&= z\left(\sum_{n=0}^N (1-z)^{n}\right) - (N+1) z(1 - z)^{N+1} = 1 - (1-z)^{N+1} - (N+1) z(1 - z)^{N+1}.
\end{align*}
This implies that $z^2P_J = z^2$ since $z^2(1-z)^{N+1}=0$, and so
\begin{gather*} 
P_J^2
= \left(\sum_{n=0}^N \sum_{m=0}^{N-n} (1 - z)^{n+m} z^2\right) P_J
= \sum_{n=0}^N \sum_{m=0}^{N-n} (1 - z)^{n+m} z^2
= P_J.
\end{gather*}
\endproof

\begin{remark}\emph{\label{PJsummationfree}
As shown in the calculation above, one could define $P_J$ as \[P_J = 1 - (1+(N+1)z_J)(1-z_J)^{N+1},\] where $N$ is the length of the longest chain in the monoid, or even $N = |S|$. For a $J$-trivial monoid, it suffices to take $P_J = 1 - (1-z_J)^{N+1}$.
} \end{remark}

\begin{lemma}\label{lemma:p-ortho}
	For all $J, K \in \calL$, with $J \not \preceq K$, {$P_J P_K = 0$.}  
\end{lemma}
\proof
Follows from Lemma \ref{lemma:z-ortho} and the fact that $P_J$ is a polynomial in $z_J$ with no constant term.
\endproof

\begin{defn}
	For each $J \in \calL$, let $$e_J :=  P_J \left( 1 - \sum_{K \succ J} e_K  \right).$$
\end{defn}

\begin{lemma}\label{e!=0}
$T_J$ occurs in $e_J$ with coefficient 1. All other terms $y$ of $e_J$ have $C(y) \succ J$. In particular, $e_J \neq 0$.
\end{lemma}

\proof
We proceed by induction. If $J$ is maximal, then $e_J = P_J$, so the statement is implied by Lemma \ref{p!=0}.

Now suppose the statement is true for all $M \succ J$. Then $e_J = P_J(1-\sum_{M \succ J}e_M)$. By induction, all terms $x$ of $e_M$ have $C(x) \succeq M \succ J$. So terms $y$ from $P_J e_M$  have $C(y) \succeq M \succ J$. The only other terms are those from $P_J$, for which the statement was proved in Lemma \ref{p!=0}.
\endproof

\begin{lemma}\label{lemma:ep}
$e_K P_J= 0$ for $K  \not\preceq J$.
\end{lemma}

\proof
The proof is by a downward induction on the semi-lattice. If $K$ is maximal, then $e_K = P_K$, so by Lemma \ref{lemma:p-ortho}, $e_K P_J= P_KP_J = 0$. 

Now suppose that for every $L\succ K$, $ e_L P_J= 0$ for $L \not\preceq J$, and we will show that $ e_K P_J=0$ for $K \not\preceq J$. We expand $e_KP_J$:
\begin{align*}
e_K P_J = P_K\left(1-\sum_{L \succ K} e_L\right)P_J = P_KP_J - \sum_{L\succ K} P_Ke_LP_J.
\end{align*}
Since $K \not\preceq J$, we have
$P_KP_J = 0$ by Lemma \ref{lemma:p-ortho}, and $e_LP_J=0$ by induction, since $L\succ K$ and $K \not \preceq J$ implies $L \not\preceq J$. 
\endproof

\begin{cor}\label{idem-e}
$e_J$ is idempotent.
\end{cor}

\proof
We expand $e_J e_J$:
\begin{align*} e_Je_J 
    &= P_J\left(1-\sum_{M\succ J}e_M\right)P_J\left(1-\sum_{M\succ J}e_M\right)
    = P_J\left(P_J-\sum_{M\succ J}e_MP_J\right)\left(1 - \sum_{M\succ J} e_M\right)\\
    &\buildrel{(1)}\over{=} P_J^2\left(1-\sum_{M \succ J} e_M\right)
    \buildrel{(2)}\over{=} P_J \left(1-\sum_{M \succ J} e_M\right)=e_J,
\end{align*}
where (1) follows from Lemma \ref{lemma:ep},
and (2) follows from Lemma \ref{lemma:p-idemp}.
\endproof

\begin{lemma}\label{ortho-e}
$e_Je_K = 0$ for $J\neq K$.
\end{lemma}

\proof
The proof is by downward induction on the lattice $\mathcal{L}$. For a maximal element $M \in \mathcal{L}$, $e_M = P_M$, so $e_M e_K = P_M P_K (1-\sum e_L) = 0$ by Lemma \ref{lemma:p-ortho}. Now suppose that for all $M\succ J$, $e_M e_K = 0$ for $M \neq K$ and we will show that $e_J e_K = 0$ for $J \neq K$. 
We expand $e_Je_K$:

\begin{equation}\label{eqn:ortho}
 e_Je_K = P_J(1-\sum_{L\succ J}e_L)e_K = P_J(e_K-\sum_{L\succ J}e_Le_K)
\end{equation}

If $K \not \succ  J$, then $\sum_{L\succ J}e_Le_K = 0 $ by our induction hypothesis, so $P_J(e_K-\sum_{L\succ J}e_Le_K) = P_J e_K = P_J P_K (1-\sum_{M\succ K}e_M) = 0$ by Lemma \ref{lemma:p-ortho}.

If $K \succ J$, then  $\sum_{L\succ J}e_Le_K = e_K$ since $e_K$ is idempotent and $e_L e_K = 0$ for $L \neq K$ by the inductive hypothesis. Therefore $e_K-\sum_{L\succ J}e_Le_K =0 $ and hence the right hand side of (\ref{eqn:ortho}) is zero.
\endproof

\begin{thm}
    The set $\{ e_J \, : \ J \in \calL \}$ is a complete system of primitive orthogonal idempotents for $\bbC S$.
\end{thm}

\proof
From \cite{S}, we know that the maximal number of such idempotents is the cardinality of $\mathcal{L}$. The rest of the claim is just Lemma \ref{e!=0}, Corollary \ref{idem-e} and Lemma \ref{ortho-e}.
\endproof

\section*{Appendix: Two examples}

We illustrate the above constructions on two examples.

\subsection*{Idempotents for the free left regular band on two generators}

Let $S$ be the left regular band freely generated by two elements $a,b$. Then $S = \{ 1, a, b, ab, ba\}$. All elements of $S$ are idempotent. Also $aba = ab$ and $bab = ba$. The lattice $\mathcal{L}$ has four elements: $\emptyset := S, \mathfrak{a} := Sa, \mathfrak{b} := Sb \textrm{ and }\mathfrak{ab} := Sab = Sba$, where $\emptyset \prec \mathfrak{a} \prec \mathfrak{ab}$ and $\emptyset \prec \mathfrak{b} \prec \mathfrak{ab}$, but $\mathfrak{a}$ and $\mathfrak{b}$ have no relation. 
We begin by computing the elements $P_J$.

$J = \emptyset$:
Neither of the generators satisfies $C(g) \preceq J$, so $T_\emptyset = 1 \in S$. $B_\emptyset = (1-a)(1-b)$. Also 
\begin{align*} B_\emptyset^2 & = (1-a)(1-b)(1-a)(1-b) = (1-a-b+ab)(1-a)(1-b) \\&= (1-a-b+ab)(1-b) = (1-a-b+ab) = B_\emptyset.
\end{align*}
Therefore $A_\emptyset = B_\emptyset = 1-a-b+ab$, so $z_\emptyset = 1-a-b+ab$ is idempotent and $$P_\emptyset = 1-a-b+ab.$$

$J = \mathfrak{a}$:
Then $C(a) \preceq \mathfrak{a}$ and $C(b) \not \preceq \mathfrak{a}$, so $T_\mathfrak{a} = a$ and $B_\mathfrak{a} = 1-b = A_\mathfrak{a}$ since $1-b$ is idempotent. Therefore $z_\mathfrak{a} = (1-b)a = a - ba$. $z_\mathfrak{a}^2 = a-ab$ and one can check that $z_\mathfrak{a}^3 = z_\mathfrak{a}^2$, so $$P_\mathfrak{a} = z_\mathfrak{a}^2 (1+ (1-z_\mathfrak{a}) + (1-z_\mathfrak{a})^2 + \dots ) = z_\mathfrak{a}^2 = a-ab.$$ One can check that $P_\mathfrak{a}$ is idempotent. 

$J = \mathfrak{b}$:
Similarly, $$P_\mathfrak{b} = b- ba.$$

$J = \mathfrak{ab}$: $C(a), C(b) \preceq \mathfrak{ab}$, so $T_\mathfrak{ab} = ab$ and $A_\mathfrak{ab} = 1$. $z_\mathfrak{ab} = ab$ is idempotent, so 
$$P_\mathfrak{ab} = ab.$$

We can now compute the idempotents $e_J$. Since $\mathfrak{ab}$ is maximal, $${e_\mathfrak{ab} = ab}.$$ 
Since $P_\mathfrak{a} e_\mathfrak{ab} = (a-ab)ab = ab - ab = 0$, $${e_\mathfrak{a}} = P_\mathfrak{a} (1- e_\mathfrak{ab}) = P_\mathfrak{a} = {a-ab}$$
and similarly, $${e_\mathfrak{b} = b-ba}.$$
Finally, note that 
$P_\emptyset e_\mathfrak{a}  = (1-a-b+ab)(a-ab)   = 0$ and similarly $P_\emptyset e_\mathfrak{b} = 0$, so that
$${e_\emptyset} = P_\emptyset (1-e_\mathfrak{a}-e_\mathfrak{b}-e_\mathfrak{ab}) = P_\emptyset - P_\emptyset e_\mathfrak{ab} = 1-a-b+ab -ab +ba = {1-a-b+ba}.$$
One can check that $\{ e_\emptyset, e_\mathfrak{a}, e_\mathfrak{b}, e_\mathfrak{ab} \}$ is a collection of mutually orthogonal idempotents.

\subsection*{Idempotents of $H^{\mathfrak{S}_5}(0)$} As mentioned above, $H^{\mathfrak{S}_5}(0)$ has generators $T_1, T_2, T_3, T_4$. In this case, the corresponding lattice $\mathcal{L}$ is the lattice of subsets of $\{ 1,2,3,4\}$. The monoid $H^{\mathfrak{S}_5}(0)$ is actually a $J$-trivial monoid, so we can use the simplified formula from Remark \ref{jtriv}.
We use the shorthand notation $T_{i_1 \cdots i_k}$ to denote the element $T_{i_1} \cdots T_{i_k}$.

If $J = \{1,2,3,4\}$, then $T_J = T_{1234}^\omega = T_{1234123121}$. Also $A_J = 1$, so $z_J = A_J T_J = T_J$.  Also, $P_J = z_J$, and since $J$ is maximal, $e_J = P_J$, so
\[
e_{\{1, 2, 3, 4\}} = T_{1234123121}.
\]

If $J = \{1,2,3\}$, then $T_J = T_{123121}$ and $A_J = 1-T_4$. Then $z_J = (1-T_4) T_{123121} = T_{123121} - T_{4123121}$. One can check that $z_J^2 = z_J$, so $P_J = z_J$. Also, one can check that $P_J$ is orthogonal to $e_{\{1,2,3,4\}}$. So $e_J = P_J$. Therefore
\[
e_{\{1, 2, 3\}} = T_{123121} - T_{4123121}.
\]

Similarly, 
\[
e_{\{2, 3, 4\}} = -T_{1234232} + T_{234232}.
\]

Now let $J = \{ 1,2,4\}$. Then $T_J = T_{1214}$ and $A_J = (1-T_3)$. Letting $z_J = A_JT_J$, one can check that $z_J (1-z_J)^2 = 0$, so $P_J = z_J (1 + (1-z_J))$. Again $P_J$ is orthogonal to $e_{\{1,2,3,4\}}$, so $e_J= P_J$. Therefore
\\

$
e_{\{1, 2, 4\}} = -T_{123423121} + T_{12343121} - T_{34121} + T_{4121}.
$
\\

Similarly, 
\\

$
e_{\{1, 3, 4\}} = -T_{123412321} + T_{12342321} - T_{23431} + T_{3431}.
$
\\

When $J = \{1,2\}$, $T_J = T_{121}$ and $A_J = (1-T_3)(1-T_4)(1-T_3)$. Then $z_J$ is already idempotent, so $P_J = z_J$. One can check that $P_J$ is already orthogonal to $e_{\{1,2,3,4\}}, e_{\{1,2,3\}}, e_{\{1,2,4\}}$. Therefore, 
\\

$
e_{\{1, 2\}} = T_{121} - T_{3121} + T_{34121} - T_{343121} - T_{4121} + T_{43121}.
$
\\

Similarly,
\\

$
e_{\{3, 4\}} = T_{12343} - T_{123431} - T_{2343} + T_{23431} + T_{343} - T_{3431}.
$
\\

If $J = \{1,3\}$, $T_J = T_1T_3$ and $A_J = (1-T_2)(1-T_4)$. One can check that $z_J (1-z_J)^2 = 0$, and $P_J = z_J (1+1-z_J)$ is idempotent. $P_J$ is orthogonal to $e_{\{1,2,3,4\}}$ and $e_{\{1,2,3\}}$, but not orthogonal to $e_{\{1,2,4\}}$. So we define $e_{\{1,3\}}= P_{\{1,3\}} (1-e_{\{1,2,4\}})$. Then 
\\

$
e_{\{1, 3\}} = -T_{123121} + T_{12321} - T_{12341231} + T_{123412321} + T_{1234231} - T_{12342321} - T_{231} + T_{2341231} - T_{23412321} + T_{31} - T_{341231} + T_{3412321} + T_{4123121} - T_{412321} + T_{4231} - T_{431}.
$
\\

Similarly, 
\\

$
e_{\{2, 4\}} = -T_{12342312} + T_{123423121} + T_{1234232} + T_{1234312} - T_{12343121} - T_{123432} + T_{2342312} - T_{23423121} - T_{234232} - T_{234312} + T_{2343121} + T_{23432} + T_{3412} - T_{342} - T_{412} + T_{42}.
$
\\

We continue in this way, constructing all of the idempotents for the algebra. For the sake of completeness, the other idempotents are:
\\

$
e_{\{2, 3\}} = -T_{1232} + T_{123412312} - T_{1234123121} + T_{232} - T_{23412312} + T_{234123121} + T_{41232} - T_{4232};
$
\\

$
e_{\{1, 4\}} = -T_{1234123121} + T_{123412321} + T_{123423121} - T_{12342321} - T_{12343121} + T_{1234321} + T_{2341} - T_{23421} - T_{341} + T_{3421} + T_{41} - T_{421};
$
\\

$
e_{\{4\}} = -T_{1234} + T_{12341} - T_{123412} + T_{1234121} + T_{12342} - T_{123421} + T_{234} - T_{2341} + T_{23412} - T_{234121} - T_{2342} + T_{23421} - T_{34} + T_{341} - T_{3412} + T_{34121} + T_{342} - T_{3421} + T_{4} - T_{41} + T_{412} - T_{4121} - T_{42} + T_{421};
$
\\

$
e_{\{3\}} = T_{123} - T_{1231} + T_{1234123} - T_{12341232} - T_{123423} + T_{1234232} - T_{23} + T_{231} - T_{234123} + T_{2341232} + T_{23423} - T_{234232} + T_{3} - T_{31} + T_{34123} - T_{341232} - T_{3423} + T_{34232} - T_{4123} + T_{41231} + T_{423} - T_{4231} - T_{43} + T_{431};
$
\\

$
e_{\{2\}} = -T_{12} + T_{12312} - T_{123121} + T_{2} - T_{2312} + T_{23121} + T_{312} - T_{32} - T_{3412} + T_{3412312} - T_{34123121} + T_{342} - T_{342312} + T_{3423121} + T_{34312} - T_{3432} + T_{412} - T_{412312} + T_{4123121} - T_{42} + T_{42312} - T_{423121} - T_{4312} + T_{432};
$
\\

$
e_{\{1\}} = T_{1} - T_{21} + T_{231} - T_{2321} - T_{2341} + T_{23421} - T_{234231} + T_{2342321} + T_{23431} - T_{234321} - T_{31} + T_{321} + T_{341} - T_{3421} + T_{34231} - T_{342321} - T_{3431} + T_{34321} - T_{41} + T_{421} - T_{4231} + T_{42321} + T_{431} - T_{4321}.
$
\\

Finally, $e_{\{\}}$ is just the signed sum of all elements, with sign determined by Coxeter length:
\\

$
e_{\{\}} = \sum_w (-1)^{\ell(w)} T_w.
$
\\

One can check (ideally not by hand!) that $\{ e_J : J \subseteq \{1,2,3,4 \}\}$ is a complete system of orthogonal idempotents.

\section*{Acknowledgements}
The authors are grateful to Tom Denton, Florent Hivert, Anne Schilling, Benjamin
Steinberg and Nicolas M. Thi\'ery for useful and open mathematical discussions. We
first learned of the equivalence between $R$-trivial monoids and weakly ordered
monoids from Thi\'ery after discussions between the authors and Denton, Hivert,
Schilling, and Thi\'ery. The proof presented in Section \ref{rtriv} was outlined
by Steinberg.

This research was facilitated by computer exploration using the open-source
mathematical software \texttt{Sage}~\cite{sage} and its algebraic
combinatorics features developed by the \texttt{Sage-Combinat}
community~\cite{Sage-Combinat}. We are especially grateful to Nicolas M. Thi\'ery and Florent Hivert for sharing their code with us.

This work is supported in part by CRC and NSERC.
It is the result of a working session at the Algebraic
Combinatorics Seminar at the Fields Institute with the active
participation of C.~Benedetti, A.~Bergeron-Brlek, Z.~Chen, H.~Heglin, D.~Mazur and M.~Zabrocki.

\small
\bibliography{idempotents}{}
\bibliographystyle{plain}
\end{document}